\documentclass[reqno,11pt]{article} 
\usepackage{amsmath, amssymb}
\usepackage{amsthm} 
\theoremstyle{plain} 
\newtheorem{theorem}{\indent\sc Theorem}[section]
\newtheorem{lemma}[theorem]{\indent\sc Lemma}

\newtheorem{proposition}[theorem]{\indent\sc Proposition}

\theoremstyle{definition} 

\newtheorem{remark}[theorem]{\indent\sc Remark}
\newtheorem{example}[theorem]{\indent\sc Example}

\makeatletter
\def\address#1#2{\begingroup
\noindent\parbox[t]{7.8cm}{%
\small{\scshape\ignorespaces#1}\par\vskip1ex
\noindent\small{\itshape E-mail address}%
\/: #2\par\vskip4ex}\hfill%
\endgroup}%
\makeatother
\title{On some theta constants and class fields} 
\author{
\textsc{Ja Kyung Koo and Dong Hwa Shin$^{*}$} 
}
\date{} 
\begin{document}

\maketitle

 \footnote{
\thanks{
The first named author was partially supported by the NRF of Korea
grant funded by MEST (2012-0000798). $^{*}$The corresponding author was
supported by Hankuk University of Foreign Studies Research Fund of
2012.} }

\begin{abstract}
We first find a sufficient condition for a product of
theta constants to be a Siegel modular function of a given even
level. And, when $K_{(2p)}$ denotes the ray class field of
$K=\mathbb{Q}(e^{2\pi i/5})$ modulo $2p$ for an odd prime $p$, we
describe a subfield of $K_{(2p)}$ generated by the special value of
certain theta constant by using Shimura's reciprocity law.
\end{abstract}

\section {Introduction}

Let $N$ ($\geq2$) be an integer and $\mathfrak{F}_N$ be the field of
meromorphic modular functions of level $N$ whose Fourier
coefficients lie in the $N$th cyclotomic field (\cite{Shi} or
\cite[$\S$6.3]{Lang}). For a vector
$\left[\begin{matrix}r\\s\end{matrix}\right]\in(1/N)\mathbb{Z}^2-\mathbb{Z}^2$
the \textit{Siegel function}
$g_{\left[\begin{smallmatrix}r\\s\end{smallmatrix}\right]}(\tau)$ is
defined on the upper half-plane by the following infinite product
\begin{equation*}
g_{\left[\begin{smallmatrix}r\\s\end{smallmatrix}\right]}(\tau)=-q^{(1/2)
(r^2-r+1/6)}e^{\pi is(r-1)}(1-q^re^{2\pi
is})\prod_{n=1}^{\infty}(1-q^{n+r}e^{2\pi is})(1-q^{n-r}e^{-2\pi
is}),
\end{equation*}
where $q=e^{2\pi i\tau}$ and $i=\sqrt{-1}$. Let
$\{m(r,s)\}_{\left[\begin{smallmatrix}r\\s\end{smallmatrix}\right]\in(1/N)\mathbb{Z}^2-\mathbb{Z}^2}$
be a family of integers such that $m(r,s)\neq 0$ only for finitely
many vectors $\left[\begin{matrix}r\\s\end{matrix}\right]$. Kubert and Lang \cite[Chapter 3, Theorem
5.3]{K-L} showed that if
$\{m(r,s)\}_{\left[\begin{smallmatrix}r\\s\end{smallmatrix}\right]}$
satisfies the quadratic relation modulo $N$, namely
\begin{eqnarray*}
&&\sum_{\left[\begin{smallmatrix}r\\s\end{smallmatrix}\right]} m(r,s)(Nr)^2\equiv\sum_{\left[\begin{smallmatrix}r\\s\end{smallmatrix}\right]}
m(r,s)(Ns)^2\equiv0\pmod{\gcd(2,N)\cdot N},\\
&&\sum_{\left[\begin{smallmatrix}r\\s\end{smallmatrix}\right]} m(r,s)(Nr)(Ns)\equiv0\pmod{N},
\end{eqnarray*}
and $12$ divides $\gcd(12,N)\cdot\sum_{\left[\begin{smallmatrix}r\\s\end{smallmatrix}\right]} m(r,s)$, then the
product $\prod_{\left[\begin{smallmatrix}r\\s\end{smallmatrix}\right]}
g_{\left[\begin{smallmatrix}r\\s\end{smallmatrix}\right]}(\tau)^{m(r,s)}$
belongs to $\mathfrak{F}_N$. In particular,
$g_{\left[\begin{smallmatrix}r\\s\end{smallmatrix}\right]}(\tau)^{12N}$
belongs to $\mathfrak{F}_N$ for any vector
$\left[\begin{matrix}r\\s\end{matrix}\right]\in(1/N)\mathbb{Z}^2-\mathbb{Z}^2$,
which depends only on
$\left[\begin{matrix}r\\s\end{matrix}\right]\pmod{\mathbb{Z}^2}$.
And, the group $\mathrm{GL}_2(\mathbb{Z}/N\mathbb{Z})$ acts on the
family
$\{g_{\left[\begin{smallmatrix}r\\s\end{smallmatrix}\right]}(\tau)^{12N}\}_
{\left[\begin{smallmatrix}r\\s\end{smallmatrix}\right]\in(1/N)\mathbb{Z}^2/\mathbb{Z}^2-
\{\left[\begin{smallmatrix}0\\0\end{smallmatrix}\right]\}}$ by the
rule
\begin{equation*}
(g_{\left[\begin{smallmatrix}r\\s\end{smallmatrix}\right]}(\tau)^{12N})^
{\alpha}=g_{{^t}\alpha\left[\begin{smallmatrix}r\\s\end{smallmatrix}\right]}(\tau)^{12N}
\quad(\alpha\in\mathrm{GL}_2(\mathbb{Z}/N\mathbb{Z})),
\end{equation*}
where ${^t}\alpha$ stands for the transpose of $\alpha$
\cite[Chapter 2, Proposition 1.3]{K-L}.
\par
Now, let $K$ be an imaginary quadratic field and 
$\mathcal{O}_K=\mathbb{Z}[\theta]$ be its ring of integers with
$\mathrm{Im}(\theta)>0$. For a positive integer $N$ we denote the
ray class field of $K$ modulo $N$ by $K_{(N)}$. The main theorem of
complex multiplication ensures that $K_{(N)}$ is generated by the
singular values $f(\theta)$ for all $f\in\mathfrak{F}_N$ which are
finite at $\theta$ \cite[Chapter 10, Corollary to Theorem 1]{Lang}.
And, by using Shimura's reciprocity law \cite[Theorem 6.31]{Shi}
Jung et al recently showed in \cite{J-K-S} that if
$K\neq\mathbb{Q}(\sqrt{-1}),\mathbb{Q}(\sqrt{-3})$ and $N\geq2$,
then $K_{(N)}$ is generated by the singular value
$g_{\left[\begin{smallmatrix}0\\1/N\end{smallmatrix}\right]}(\theta)^{12N}$
over $K$.
\par
In this paper we shall attempt to find higher dimensional analogues
of these results. Siegel modular functions of level $N$ ($\geq1$)
defined on the Siegel upper half-space $\mathfrak{H}_g$ ($g\geq2$)
are certain multi-variable functions which generalize meromorphic
modular functions of one variable ($\S$\ref{Smf}). As in the case of
modular functions, the action of the general symplectic group
$\mathrm{GSp}_{2g}(\mathbb{Z}/N\mathbb{Z})$ on the Siegel modular
functions of level $N$ was investigated by Shimura (Proposition
\ref{Groupaction}). If $\mathbf{r},\mathbf{s}\in(1/N)\mathbb{Z}^g$,
then the theta constant
\begin{equation*}
\Phi_{\left[\begin{smallmatrix}\mathbf{r}\\\mathbf{s}\end{smallmatrix}\right]}(Z)=
\frac{\sum_{\mathbf{x}\in\mathbb{Z}^g}
e({^t}(\mathbf{x}+\mathbf{r})Z(\mathbf{x}+\mathbf{r})/2
+{^t}(\mathbf{x}+\mathbf{r})\mathbf{s})}{
\sum_{\mathbf{x}\in\mathbb{Z}^g} e({^t}\mathbf{x}Z\mathbf{x}/2)}
\end{equation*}
is a typical example of Siegel modular functions (of level $2N^2$)
($\S$\ref{quotient}). We shall first give a sufficient
condition for a product of theta constants to be a Siegel modular
function of level $N$ when $N$ is even (Theorem \ref{modularity}).
And, we shall further show that certain subgroup of
$\mathrm{GSp}_{2g}(\mathbb{Z}/N\mathbb{Z})$ acts on the family
$\{\Phi_{\left[\begin{smallmatrix}\mathbf{r}\\\mathbf{s}\end{smallmatrix}\right]}(Z)^{2N^2}\}
_{\mathbf{r},\mathbf{s}\in(1/N)\mathbb{Z}^g/\mathbb{Z}^g}$ in a natural way, namely
\begin{equation*}
(\Phi_{\left[\begin{smallmatrix}\mathbf{r}\\\mathbf{s}\end{smallmatrix}\right]}(Z)^{2N^2})
^\alpha=\Phi_{{^t}\alpha\left[\begin{smallmatrix}\mathbf{r}\\\mathbf{s}\end{smallmatrix}\right]}(Z)^{2N^2}
\end{equation*}
(Theorem \ref{family}).
\par
On the other hand, let $K$ be a CM-field, $K^*$ be its reflex field
and $Z_0$ be the associated CM-point ($\S$\ref{Srl}). The theory of
complex multiplication for polarized abelian varieties of higher
dimension developed by Shimura claims that if $f$ is a Siegel
modular function that is finite at $Z_0$, then the special value
$f(Z_0)$ lies in some abelian extension of $K^*$. Furthermore,
Shimura's reciprocity law describes Galois actions on $f(Z_0)$ in
terms of actions of general symplectic groups on $f$ (Proposition
\ref{reciprocity}). Here, we focus on the case where
$K=\mathbb{Q}(e^{2\pi i/5})$. For an odd prime $p$ let $K_{(2p)}$
and $K_{(2p^2)}$ be the ray class fields of $K$ modulo $2p$ and
$2p^2$, respectively. Komatsu
 considered in \cite{Komatsu} certain intermediate field $L$ of
$K_{(2p^2)}/K_{(2p)}$ with $[L:K_{(2p)}]=p^3$ and provided a normal
basis of $L$ over $K_{(2p)}$. Unlike Komatsu's work, however, we
shall examine the field
$K(\Phi_{\left[\begin{smallmatrix}1/p\\0\\0\\0\end{smallmatrix}\right]}(Z_0)^{2p^2})$
as a subfield of $K_{(2p)}$ (Theorems \ref{howlarge} and
\ref{belong}). To this end we shall utilize transformation formulas
of theta functions (Propositions \ref{transl} and \ref{transf})
together with Shimura's reciprocity law.
\par
And, we shall also present some ideas of combining two generators of
an abelian extension to get a primitive generator (Theorems
\ref{primitive1} and \ref{primitive2}).

\section {Siegel modular forms}\label{Smf}

We shall introduce necessary facts about Siegel modular forms, and
explain actions of general symplectic groups on the Siegel modular
functions whose Fourier coefficients lie in some cyclotomic fields.
\par
Let $g$ ($\geq2$) be a positive integer and
\begin{equation*}
J=\left[\begin{matrix}0 & -I_g\\I_g & 0
\end{matrix}\right].
\end{equation*}
Given a commutative ring $R$ with unity we let
\begin{equation*}
\mathrm{GSp}_{2g}(R)=\{\alpha\in\mathrm{Mat}_{2g}(R)~|~ {^t}\alpha
J\alpha =\nu J ~\textrm{for some}~\nu\in R^\times\}.
\end{equation*}
Considering $\nu$ as a homomorphism $\mathrm{GSp}_{2g}(R)\rightarrow
R^\times$ we denote its kernel by $\mathrm{Sp}_{2g}(R)$, namely
\begin{equation*}
\mathrm{Sp}_{2g}(R)=\{\alpha\in\mathrm{Mat}_{2g}(R)~|~ {^t}\alpha
J\alpha=J \}.
\end{equation*}
We further define a homomorphism
$\iota:R^\times\rightarrow\mathrm{GSp}_{2g}(R)$ by
\begin{equation*}
\iota(a)=\left[\begin{matrix} I_g & 0 \\
0 & a^{-1}I_g
\end{matrix}\right].
\end{equation*}
One can then readily show that $\nu(\iota(a))=a^{-1}$.
\par
The \textit{Siegel upper half-space} $\mathfrak{H}_g$ is defined by
\begin{equation*}
\mathfrak{H}_g =\{Z\in
\mathrm{Mat}_g(\mathbb{C})~|~{^t}Z=Z,~\mathrm{Im}(Z)~\textrm{is
positive definite}\}.
\end{equation*}
Then, it is well-known that $\mathrm{Sp}_{2g}(\mathbb{Z})$ acts on
$\mathfrak{H}_g$ by
\begin{equation*}
\left[\begin{matrix}A & B\\C&D\end{matrix}\right](Z)
=(AZ+B)(CZ+D)^{-1},
\end{equation*}
where $A,B,C,D$ are $g\times g$ block matrices \cite[$\S$1,
Proposition 1]{Klingen}. Let $N$ ($\geq1$) and $k$ be integers, and
define the group
\begin{equation*}
\Gamma(N)=\{\gamma\in\mathrm{Sp}_{2g}(\mathbb{Z})~|~ \gamma\equiv
I_{2g}\pmod{N}\}.
\end{equation*}
A holomorphic function $f:\mathfrak{H}_g\rightarrow\mathbb{C}$ is
called a \textit{Siegel modular form} of weight $k$ and level $N$,
if
\begin{equation*}
f(\gamma(Z))=\det(CZ+D)^k f(Z)\quad\textrm{for every}~
\gamma=\left[\begin{matrix} A&B\\C&D
\end{matrix}\right]\in\Gamma(N).
\end{equation*}
For $z\in\mathbb{C}$, we set
\begin{equation*}
e(z)=e^{2\pi iz}.
\end{equation*}
As a consequence of K\"{o}echer's principle, a Siegel modular form
$f$ can be written as
\begin{equation*}
f(Z)=\sum_{\xi}c(\xi)e(\mathrm{tr}(\xi Z)/N)\quad(c(\xi)\in
\mathbb{C}),
\end{equation*}
where $\xi$ runs over all $g\times g$ positive semi-definite
symmetric matrices over half integers with integral diagonal entries
\cite[$\S$4, Theorem 1]{Klingen}. This expansion is called the
\textit{Fourier expansion} of $f$ with \textit{Fourier coefficients}
$c(\xi)$. Note that if
$\xi=\left[\begin{matrix}\xi_{jk}\end{matrix}\right]_{1\leq j,k\leq
g}$ and $Z=\left[\begin{matrix}Z_{jk}\end{matrix}\right]_{1\leq
j,k\leq g}$, then
\begin{equation*}
\mathrm{tr}(\xi Z)=\sum_{j=1}^g\xi_{jj}Z_{jj} +2\sum_{1\leq j<k\leq
g}\xi_{jk}Z_{jk},
\end{equation*}
from which it follows that
\begin{equation*}
f(Z)=\sum_{\xi}c(\xi) \bigg(\prod_{j=1}^g
e(\xi_{jj}Z_{jj}/N)\prod_{1\leq j<k\leq g} e(2\xi_{jk}Z_{jk}/N)
\bigg).
\end{equation*}
Letting $\zeta_N=e(1/N)$ we consider the field
\begin{equation*}
\mathcal{F}_N=\left\{
\begin{array}{lll}g_1/g_2 &\bigg|&
\begin{array}{l}g_1~\textrm{and}~g_2~(\neq0)~\textrm{are Siegel
modular forms of the same weight such that }\\
\textrm{$g_1/g_2$ is invariant under $\Gamma(N)$ and its Fourier
coefficients lie in $\mathbb{Q}(\zeta_N)$}\end{array}
\end{array}\right\}.
\end{equation*}

\begin{proposition}\label{Groupaction}
\begin{itemize}
\item[\textup{(i)}] $\iota((\mathbb{Z}/N\mathbb{Z})^\times)$ acts on
$\mathcal{F}_N$ as follows: If $a\in
(\mathbb{Z}/N\mathbb{Z})^\times$ and $f=\sum_\xi
c(\xi)e(\mathrm{tr}(\xi Z)/N)\in\mathcal{F}_N$, then
\begin{equation*}
f^{\iota(a)}=\sum_\xi c(\xi)^{\rho(a)^{-1}}e(\mathrm{tr}(\xi Z)/N),
\end{equation*}
where $\rho(a)$ is an endomorphism of $\mathbb{Q}(\zeta_N)$ induced
from the map $\zeta_N\mapsto\zeta_N^a$.
\item[\textup{(ii)}] $\mathrm{Sp}_{2g}(\mathbb{Z})$ acts on
$\mathcal{F}_N$ by compositions, that is, if
$\gamma\in\mathrm{Sp}_{2g}(\mathbb{Z})$ and $f\in \mathcal{F}_N$,
then
\begin{equation*}
f^\gamma=f\circ\gamma.
\end{equation*}
\item[\textup{(iii)}] $\mathrm{GSp}_{2g}(\mathbb{Z}/N\mathbb{Z})$ acts on $\mathcal{F}_N$ as follows: Let
$\alpha\in\mathrm{GSp}_{2g}(\mathbb{Z}/N\mathbb{Z})$ and
$f\in\mathcal{F}_N$. Set
$a=\nu(\alpha)\in(\mathbb{Z}/N\mathbb{Z})^\times$ and
$\gamma=\iota(a)\alpha\in\mathrm{Sp}_{2g}(\mathbb{Z}/N\mathbb{Z})$.
Lift $\gamma$ to $\gamma_0\in\mathrm{Sp}_{2g}(\mathbb{Z})$. Then,
\begin{equation*}
f^\alpha=(f^{\iota(a)^{-1}})^{\gamma_0}.
\end{equation*}
\end{itemize}
\end{proposition}
\begin{proof}
See \cite[$\S$1]{Shimura2}.
\end{proof}

\section{Theta functions}

We shall briefly review fundamental transformation formulas of theta
functions.
\par
Let $g$ ($\geq2$) be an integer, $\mathbf{u}\in\mathbb{C}^g$,
$Z\in\mathfrak{H}_g$ and $\mathbf{r},\mathbf{s}\in\mathbb{R}^g$. We
define a (classical) \textit{theta function} by
\begin{equation}\label{theta}
\Theta(\mathbf{u},Z;\mathbf{r},\mathbf{s})=
\sum_{\mathbf{x}\in\mathbb{Z}^g}
e({^t}(\mathbf{x}+\mathbf{r})Z(\mathbf{x}+\mathbf{r})/2
+{^t}(\mathbf{x}+\mathbf{r})(\mathbf{u}+\mathbf{s})),
\end{equation}
which is a holomorphic function on $Z$. Since $\mathbf{x}$ can be
replaced by $-\mathbf{x}$ in the above summation, we get the
relation
\begin{equation}\label{sign}
\Theta(-\mathbf{u},Z;-\mathbf{r},-\mathbf{s})
=\Theta(\mathbf{u},Z;\mathbf{r},\mathbf{s}).
\end{equation}

\begin{proposition}\label{transl}
If $\mathbf{a},\mathbf{b}\in\mathbb{Z}^g$, then we have the
translation formula
\begin{equation*}
\Theta(\mathbf{u},Z;\mathbf{r}+\mathbf{a},\mathbf{s}+\mathbf{b})=
e({^t}\mathbf{r}\mathbf{b})\Theta(\mathbf{u},Z;\mathbf{r},\mathbf{s}).
\end{equation*}
\end{proposition}
\begin{proof}
See \cite[p.676 (13)]{Shimura3}.
\end{proof}

For a square matrix $\alpha$ we denote by $\{\alpha\}$ the column
vector whose components are the diagonal elements of $\alpha$.

\begin{proposition}\label{transf}
Let $\gamma=\left[\begin{matrix}A&B\\C&D
\end{matrix}\right]\in\mathrm{Sp}_{2g}(\mathbb{Z})$
such that $\{{^t}AC\},\{{^t}BD\}\in2\mathbb{Z}^g$. We get the
transformation formula
\begin{eqnarray*}
&&\Theta({^t}(CZ+D)^{-1}\mathbf{u},\gamma(Z);\mathbf{r},\mathbf{s})\\
&=& \lambda_\gamma
e(({^t}\mathbf{r}\mathbf{s}-{^t}\mathbf{r}'\mathbf{s}')/2)
\det(CZ+D)^{1/2}e(({^t}\mathbf{u}(CZ+D)^{-1}C\mathbf{u})/2)\Theta(\mathbf{u},Z;\mathbf{r}',\mathbf{s}'),
\end{eqnarray*}
where $\lambda_\gamma$ is a constant of absolute value $1$ depending
on $\gamma$ and the choice of the branch of $\det(CZ+D)^{1/2}$, and
$\left[\begin{matrix}
\mathbf{r}'\\
\mathbf{s}'
\end{matrix}\right]
={^t}\gamma \left[\begin{matrix}
\mathbf{r}\\
\mathbf{s}
\end{matrix}\right]$.
\end{proposition}
\begin{proof}
See \cite[Proposition 1.3]{Shimura3}.
\end{proof}

And, let
\begin{equation*}
\Sigma_-=\bigg\{\left[\begin{matrix}\mathbf{r}\\\mathbf{s}\end{matrix}\right]\in\mathbb{Q}^{2g}~|~\mathbf{r},\mathbf{s}\in
(1/2)\mathbb{Z}^g~\textrm{and}~e(2{^t}\mathbf{r}\mathbf{s})=-1\}.
\end{equation*}

\begin{proposition}\label{Igusazero}
Let $\mathbf{r},\mathbf{s}\in\mathbb{Q}^g$. Then,
$\Theta(\mathbf{0},Z;\mathbf{r},\mathbf{s})$ represents the zero
function on $Z$ if and only if
$\left[\begin{matrix}\mathbf{r}\\\mathbf{s}\end{matrix}\right]\in\Sigma_-$.
\end{proposition}
\begin{proof}
See \cite[Theorem 2]{Igusa}.
\end{proof}

\section {Modularity of theta constants}\label{quotient}

In this section we shall find a sufficient condition
for a product of theta constants to be a Siegel modular function of
a given even level.
\par
 Let $N$ and $g$ ($\geq2$) be positive integers and
$\mathbf{r},\mathbf{s}\in(1/N)\mathbb{Z}^g$. We define a
\textit{theta constant} by
\begin{equation*}
\Phi_{\left[\begin{smallmatrix}\mathbf{r}\\\mathbf{s}\end{smallmatrix}\right]}(Z)=
\frac{\Theta(\mathbf{0},Z;\mathbf{r},\mathbf{s})}{\Theta(\mathbf{0},Z;\mathbf{0},\mathbf{0})}
\quad(Z\in\mathfrak{H}_g),
\end{equation*}
which is a nonzero function whenever we assume
$\left[\begin{matrix}\mathbf{r}\\\mathbf{s}\end{matrix}\right]\not\in\Sigma_-$
by Proposition \ref{Igusazero}. It belongs to $\mathcal{F}_{2N^2}$
by Propositions \ref{transl}, \ref{transf} and the definition
(\ref{theta}) (or \cite[Proposition 7]{Shimura3}). We get by
(\ref{sign}) that
\begin{equation}\label{minus}
\Phi_{\left[\begin{smallmatrix}-\mathbf{r}\\-\mathbf{s}\end{smallmatrix}\right]}(Z)=
\Phi_{\left[\begin{smallmatrix}\mathbf{r}\\\mathbf{s}\end{smallmatrix}\right]}(Z).
\end{equation}

\begin{lemma}\label{coefficient}
Let $a\in(\mathbb{Z}/2N^2\mathbb{Z})^\times$. The action of
$\iota(a^{-1})=
\left[\begin{matrix}I_g & 0\\
0 &
aI_g\end{matrix}\right]\in\mathrm{GSp}_{2g}(\mathbb{Z}/2N^2\mathbb{Z})$
on
$\Phi_{\left[\begin{smallmatrix}\mathbf{r}\\\mathbf{s}\end{smallmatrix}\right]}(Z)$
can be described as
\begin{equation*}
\Phi_{\left[\begin{smallmatrix}\mathbf{r}\\\mathbf{s}\end{smallmatrix}\right]}(Z)^{\iota(a^{-1})}
=\Phi_{\left[\begin{smallmatrix}\mathbf{r}\\a\mathbf{s}\end{smallmatrix}\right]}(Z)
=\Phi_{{^t}\iota(a^{-1})\left[\begin{smallmatrix}\mathbf{r}\\\mathbf{s}\end{smallmatrix}\right]}(Z).
\end{equation*}
\end{lemma}
\begin{proof}
We see from Proposition \ref{transl} that
$\Phi_{\left[\begin{smallmatrix}\mathbf{r}\\a\mathbf{s}\end{smallmatrix}\right]}(Z)$
and
$\Phi_{{^t}\iota(a^{-1})\left[\begin{smallmatrix}\mathbf{r}\\\mathbf{s}\end{smallmatrix}\right]}(Z)$
are well-defined. And, it follows from the definition (\ref{theta})
that
\begin{eqnarray*}
\Phi_{\left[\begin{smallmatrix}\mathbf{r}\\\mathbf{s}\end{smallmatrix}\right]}(Z)^{\iota(a^{-1})}
&=&\bigg(\frac{\sum_{\mathbf{x}\in\mathbb{Z}^g}
e({^t}(\mathbf{x}+\mathbf{r})\mathbf{s}) e({^t}(\mathbf{x}+
\mathbf{r})Z(\mathbf{x}+\mathbf{r})/2)}{\sum_{\mathbf{x}\in\mathbb{Z}^g}
e(({^t}\mathbf{x}Z\mathbf{x})/2)}\bigg)^{\iota(a^{-1})}\\
&=&\frac{\sum_{\mathbf{x}\in\mathbb{Z}^g}
e({^t}(\mathbf{x}+\mathbf{r})a\mathbf{s}) e({^t}(\mathbf{x}+
\mathbf{r})Z(\mathbf{x}+\mathbf{r})/2)}{\sum_{\mathbf{x}\in\mathbb{Z}^g}
e(({^t}\mathbf{x}Z\mathbf{x})/2)}\quad\textrm{by Proposition
\ref{Groupaction}(i)}\\
&=&\Phi_{\left[\begin{smallmatrix}\mathbf{r}\\a\mathbf{s}\end{smallmatrix}\right]}(Z)
\\&=&\Phi_{{^t}\iota(a^{-1})\left[\begin{smallmatrix}\mathbf{r}\\\mathbf{s}\end{smallmatrix}\right]}(Z).
\end{eqnarray*}
\end{proof}

\begin{lemma}\label{composition}
For even $N$, let $\gamma=I_{2g}+N\left[\begin{matrix} A_0 &
B_0\\C_0&D_0
\end{matrix}\right]\in\Gamma(N)$ with $A_0,B_0,C_0,D_0\in\mathrm{Mat}_g(\mathbb{Z})$. Then we
have
\begin{eqnarray*}
\Phi_{\left[\begin{smallmatrix}\mathbf{r}\\\mathbf{s}\end{smallmatrix}\right]}(\gamma(Z))&=&e\bigg(-\frac{1}{2N}{^t}(N\mathbf{r})(-{^t}B_0+NA_0{^t}B_0)(N\mathbf{r})
-\frac{1}{2N}{^t}(N\mathbf{s})(C_0+NC_0{^t}D_0)(N\mathbf{s})\\
&&-\frac{1}{N}{^t}(N\mathbf{r})
(A_0+(N/2)(A_0{^t}D_0+{^t}D_0A_0+B_0{^t}C_0-{^t}B_0C_0)(N\mathbf{s})
\bigg)
\Phi_{\left[\begin{smallmatrix}\mathbf{r}\\\mathbf{s}\end{smallmatrix}\right]}(Z).
\end{eqnarray*}
\end{lemma}
\begin{proof}
We obtain from the relation ${^t}\gamma J\gamma=J$ that
\begin{equation*}
\left[\begin{matrix}
* & -I_g-N{^t}A_0-ND_0-(N{^t}A_0)(ND_0)+(N{^t}C_0)(NB_0)\\ * & *
\end{matrix}\right]=\left[\begin{matrix}
0 & -I_g\\I_g & 0
\end{matrix}\right],
\end{equation*}
which gives rise to
\begin{equation}\label{even}
D_0=-{^t}A_0-N{^t}A_0D_0+N{^t}C_0B_0.
\end{equation}
We then derive that
\begin{eqnarray*}
\Phi_{\left[\begin{smallmatrix}\mathbf{r}\\\mathbf{s}\end{smallmatrix}\right]}(\gamma(Z))
&=&\frac{\Theta(\mathbf{0},\gamma(Z);\mathbf{r},\mathbf{s})}
{\Theta(\mathbf{0},\gamma(Z);\mathbf{0},\mathbf{0})}\\
&=&\frac{e(({^t}\mathbf{r}\mathbf{s}-{^t}\mathbf{r}'\mathbf{s}')/2)
\Theta(\mathbf{0},Z;\mathbf{r}',\mathbf{s}')}
{\Theta(\mathbf{0},Z;\mathbf{0},\mathbf{0})}\quad\textrm{where
$\left[\begin{matrix}\mathbf{r}'\\\mathbf{s}'\end{matrix}\right]=
{^t}\gamma\left[\begin{matrix}\mathbf{r}\\\mathbf{s}\end{matrix}\right]$,
by Proposition \ref{transf}}\\
&=&
e\bigg(-\frac{1}{2N}{^t}(N\mathbf{r})({^t}B_0+NA_0{^t}B_0)(N\mathbf{r})
-\frac{1}{2N}{^t}(N\mathbf{s})(C_0+NC_0{^t}D_0)(N\mathbf{s})\\
&&-\frac{1}{2N}{^t}(N\mathbf{r})(A_0+{^t}D_0+NA_0{^t}D_0+NB_0{^t}C_0)(N\mathbf{s})
\bigg)
\Phi_{\left[\begin{smallmatrix}\mathbf{r}+{^t}A_0(N\mathbf{r})+{^t}
C_0(N\mathbf{s})\\\mathbf{s}+{^t}B_0(N\mathbf{r})+
{^t}D_0(N\mathbf{s})\end{smallmatrix}\right]}(Z)\\
&=&e\bigg(-\frac{1}{2N}{^t}(N\mathbf{r})(-{^t}B_0+NA_0{^t}B_0)(N\mathbf{r})
-\frac{1}{2N}{^t}(N\mathbf{s})(C_0+NC_0{^t}D_0)(N\mathbf{s})\\
&&-\frac{1}{2N}{^t}(N\mathbf{r})(A_0-{^t}D_0+NA_0{^t}D_0+NB_0{^t}C_0)(N\mathbf{s})
\bigg)
\Phi_{\left[\begin{smallmatrix}\mathbf{r}\\\mathbf{s}\end{smallmatrix}\right]}(Z)\quad\textrm{by Proposition \ref{transl}}\\
&=&e\bigg(-\frac{1}{2N}{^t}(N\mathbf{r})(-{^t}B_0+NA_0{^t}B_0)(N\mathbf{r})
-\frac{1}{2N}{^t}(N\mathbf{s})(C_0+NC_0{^t}D_0)(N\mathbf{s})\\
&&-\frac{1}{N}{^t}(N\mathbf{r})
(A_0+(N/2)(A_0{^t}D_0+{^t}D_0A_0+B_0{^t}C_0-{^t}B_0C_0)(N\mathbf{s})
\bigg)
\Phi_{\left[\begin{smallmatrix}\mathbf{r}\\\mathbf{s}\end{smallmatrix}\right]}(Z)\quad\textrm{by
(\ref{even})}.
\end{eqnarray*}
\end{proof}

\begin{theorem}\label{modularity}
For even $N$, let
$\{m(\mathbf{r},\mathbf{s})\}_{\mathbf{r},\mathbf{s}}$, where
$\mathbf{r},\mathbf{s}\in(1/N)\mathbb{Z}^g$ such that
$\left[\begin{matrix}\mathbf{r}\\\mathbf{s}\end{matrix}\right]\not\in\Sigma_-$,
be a family of integers such that $m(\mathbf{r},\mathbf{s})=0$
except finitely many pairs of $\mathbf{r},\mathbf{s}$. Consider the
following product
\begin{equation*}
\Phi(Z)=\prod_{\mathbf{r},\mathbf{s}}
\Phi_{\left[\begin{smallmatrix}\mathbf{r}\\\mathbf{s}\end{smallmatrix}\right]}(Z)^{m(\mathbf{r},\mathbf{s})}.
\end{equation*}
Then, $\Phi(Z)$ belongs to $\mathcal{F}_N$ if the family
$\{m(\mathbf{r},\mathbf{s})\}_{\mathbf{r},\mathbf{s}}$ satisfies the
condition
\begin{equation}\label{assumption}
\left\{\begin{array}{l}
\displaystyle\sum_{\mathbf{r},\mathbf{s}}m(\mathbf{r},\mathbf{s})
(N\mathbf{r}_j)(N\mathbf{r}_k) \equiv
\sum_{\mathbf{r},\mathbf{s}}m(\mathbf{r},\mathbf{s})
(N\mathbf{s}_j)(N\mathbf{s}_k)\equiv
0\pmod{2N}\quad(1\leq j,k\leq g),\vspace{0.2cm}\\
\displaystyle\sum_{\mathbf{r},\mathbf{s}}m(\mathbf{r},\mathbf{s})(N\mathbf{r}_j)(N\mathbf{s}_k)\equiv
0\pmod{N}\quad(1\leq j,k\leq g),
\end{array}\right.
\end{equation}
where
$\mathbf{r}=\left[\begin{matrix}\mathbf{r}_1\\\vdots\\\mathbf{r}_g\end{matrix}\right]$
and
$\mathbf{s}=\left[\begin{matrix}\mathbf{s}_1\\\vdots\\\mathbf{s}_g\end{matrix}\right]$.
\end{theorem}
\begin{proof}
If $\gamma\in\Gamma(N)$, then we have
\begin{eqnarray}
\Phi(\gamma(Z))&=&\prod_{\mathbf{r},\mathbf{s}}
\bigg(\frac{\Theta(\mathbf{0},\gamma(Z);\mathbf{r},\mathbf{s})}
{\Theta(\mathbf{0},\gamma(Z);\mathbf{0},\mathbf{0})}\bigg)^{m(\mathbf{r},\mathbf{s})}\nonumber\\
&=&e\bigg(
-\frac{1}{2N}\sum_{\mathbf{r},\mathbf{s}}m(\mathbf{r},\mathbf{s})
{^t}(N\mathbf{r})(-{^t}B_0+NA_0{^t}B_0)(N\mathbf{r})\nonumber\\
&&-\frac{1}{2N}\sum_{\mathbf{r},\mathbf{s}}
m(\mathbf{r},\mathbf{s}){^t}(N\mathbf{s})(C_0+NC_0{^t}D_0)(N\mathbf{s})\nonumber\\
&&-\frac{1}{N}\sum_{\mathbf{r},\mathbf{s}}
m(\mathbf{r},\mathbf{s}){^t}(N\mathbf{r})(A_0+(N/2)(A_0{^t}D_0
+{^t}D_0A_0+B_0{^t}C_0 -{^t}B_0C_0))(N\mathbf{s})
\bigg)\Phi(Z)\nonumber\\
&&\textrm{where}~\left[\begin{matrix} A_0 & B_0\\C_0&D_0
\end{matrix}\right]=\displaystyle\frac{1}{N}(\gamma-I_{2g})
\in\mathrm{Mat}_{2g}(\mathbb{Z}),~\textrm{by Lemma
\ref{composition}}.\label{derive}
\end{eqnarray}
Now, for every pair of integers $j,k$ with $1\leq j,k\leq g$ let
$E_{jk}$ be the $g\times g$ matrix whose entries are all zeros
except for the $(j,k)$th entry which is $1$. One can then easily see
that $\mathrm{Mat}_g(\mathbb{Z})$ is generated by $E_{jk}$ ($1\leq
j,k\leq g$) as $\mathbb{Z}$-module, and if
$\mathbf{u}=\left[\begin{matrix}\mathbf{u}_1\\\vdots\\\mathbf{u}_g\end{matrix}\right]$,
$\mathbf{v}=\left[\begin{matrix}\mathbf{v}_1\\\vdots\\\mathbf{v}_g\end{matrix}\right]\in\mathbb{Z}^g$,
then ${^t}\mathbf{u}E_{jk}\mathbf{v}=\mathbf{u}_j\mathbf{v}_k$.
\par
Assume first that the family
$\{m(\mathbf{r},\mathbf{s})\}_{\mathbf{r},\mathbf{s}}$ satisfies the
condition (\ref{assumption}). Then the above observation leads to
$\Phi(\gamma(\tau))=\Phi(\tau)$. On the other hand, since $\Phi(Z)$
belongs to $\mathcal{F}_{2N^2}$, its Fourier coefficients of
$\Phi(Z)$ lie in $\mathbb{Q}(\zeta_{2N^2})$. However, at this stage
we have to show that the coefficients actually lie in
$\mathbb{Q}(\zeta_N)$. To this end, let $a$ be an integer such that
$a\equiv1\pmod{N}$, which can be written as $a=1+cN$ for some
integer $c$. Regarding $\iota$ as a map on
$(\mathbb{Z}/2N^2\mathbb{Z})^\times$ we get that
\begin{eqnarray*}
\Phi(Z)^{\iota(a)^{-1}}&=&\prod_{\mathbf{r},\mathbf{s}}
(\Phi_{\left[\begin{smallmatrix}\mathbf{r}\\\mathbf{s}\end{smallmatrix}\right]}(Z)^{\iota(a)^{-1}})^{m(\mathbf{r},\mathbf{s})}\\
&=& \prod_{\mathbf{r},\mathbf{s}}
\Phi_{\left[\begin{smallmatrix}\mathbf{r}\\a\mathbf{s}\end{smallmatrix}\right]}(Z)^{m(\mathbf{r},\mathbf{s})}
\quad\textrm{by Lemma \ref{coefficient}}\\
&=& \prod_{\mathbf{r},\mathbf{s}}
\Phi_{\left[\begin{smallmatrix}\mathbf{r}\\\mathbf{s}+cN\mathbf{s}\end{smallmatrix}\right]}(Z)^{m(\mathbf{r},\mathbf{s})}\\
&=&\prod_{\mathbf{r},\mathbf{s}}
(e({^t}\mathbf{r}cN\mathbf{s})\Phi_{\left[\begin{smallmatrix}\mathbf{r}
\\\mathbf{s}\end{smallmatrix}\right]}(Z))^{m(\mathbf{r},\mathbf{s})}
\quad\textrm{by Proposition \ref{transl}}\\
&=&e\bigg( \frac{c}{N}\sum_{\mathbf{r},\mathbf{s}}
m(\mathbf{r},\mathbf{s}){^t}(N\mathbf{r})(N\mathbf{s}) \bigg)\Phi(Z)\\
&=&e\bigg( \frac{c}{N}\sum_{j=1}^g\sum_{\mathbf{r},\mathbf{s}}
m(\mathbf{r},\mathbf{s})(N\mathbf{r}_j)(N\mathbf{s}_j) \bigg)\Phi(Z)\\
&=&\Phi(Z)\quad\textrm{by the condition (\ref{assumption})},
\end{eqnarray*}
which ensures that Fourier coefficients of $\Phi(Z)$ lie in
$\mathbb{Q}(\zeta_N)$ as desired. Therefore we conclude that
$\Phi(Z)$ is in $\mathcal{F}_N$.
\end{proof}

\section {Family of theta constants}

Let $N$ be an even positive integer and
$\mathbf{r},\mathbf{s}\in(1/N)\mathbb{Z}^g$ ($g\geq2$). By Theorem
\ref{modularity},
$\Phi_{\left[\begin{smallmatrix}\mathbf{r}\\\mathbf{s}\end{smallmatrix}\right]}(Z)^{2N}$
belongs to $\mathcal{F}_{N}$. In this section we shall show that a
subgroup of $\mathrm{GSp}_{2g}(\mathbb{Z}/N\mathbb{Z})$ has a
natural action on the family
$\{\Phi_{\left[\begin{smallmatrix}\mathbf{r}\\\mathbf{s}\end{smallmatrix}\right]}(Z)^{2N^2}\}
_{\mathbf{r},\mathbf{s}\in(1/N)\mathbb{Z}^g/\mathbb{Z}^g}$.

\begin{lemma}\label{determine}
Let $M$ and $\ell$ be divisors of $N$. If
$\mathbf{r}\in(1/M)\mathbb{Z}^g$, then
$\Phi_{\left[\begin{smallmatrix}\mathbf{r}\\\mathbf{s}\end{smallmatrix}\right]}(Z)^\ell$
is determined by $\mathbf{r}\pmod{\mathbb{Z}^g}$ and $\mathbf{s}
\pmod{(M/\gcd(M,\ell))\mathbb{Z}^g}$.
\end{lemma}
\begin{proof}
If $\mathbf{a}\in\mathbb{Z}^g$ and $\mathbf{b}\in
(M/\gcd(M,\ell))\mathbb{Z}^g$, then we see by Proposition
\ref{transl} that
\begin{equation*}
\Phi_{\left[\begin{smallmatrix}\mathbf{r}+\mathbf{a}\\\mathbf{s}
+\mathbf{b}\end{smallmatrix}\right]}(Z)^\ell=(e({^t}\mathbf{r}\mathbf{b})
\Phi_{\left[\begin{smallmatrix}\mathbf{r}\\\mathbf{s}\end{smallmatrix}\right]}(Z))^\ell
=\Phi_{\left[\begin{smallmatrix}\mathbf{r}\\\mathbf{s}\end{smallmatrix}\right]}(Z)^\ell.
\end{equation*}
\end{proof}

Let
\begin{eqnarray*}
S_{N}&=&\bigg\{\left[\begin{matrix}A&B\\C&D\end{matrix}\right]\in\mathrm{Sp}_{2g}(\mathbb{Z}/N\mathbb{Z})~|~
\{{^t}AC\}\equiv\{{^t}BD\}\equiv\mathbf{0}\pmod{2}\bigg\},\\
G_{N}&=&\langle \iota((\mathbb{Z}/N\mathbb{Z})^\times), S_N\rangle,
\end{eqnarray*}
which are subgroups of $\mathrm{GSp}_{2g}(\mathbb{Z}/N\mathbb{Z})$
\cite[$\S$27.6]{Shimura}. One can then readily show that
\begin{equation*}
G_N=\bigg\{\left[\begin{matrix}A&B\\C&D\end{matrix}\right]\in\mathrm{GSp}_{2g}(\mathbb{Z}/N\mathbb{Z})~|~
\{{^t}AC\}\equiv\{{^t}BD\}\equiv\mathbf{0}\pmod{2}\bigg\}.
\end{equation*}

\begin{theorem}\label{family}
If $\alpha\in G_N$, then we have
\begin{equation*}
(\Phi_{\left[\begin{smallmatrix}\mathbf{r}\\\mathbf{s}\end{smallmatrix}\right]}(Z)^{2N^2}
)^\alpha=
\Phi_{{^t}\alpha\left[\begin{smallmatrix}\mathbf{r}\\\mathbf{s}\end{smallmatrix}\right]}(Z)^{2N^2}.
\end{equation*}
\end{theorem}
\begin{proof}
If $\alpha\in\iota((\mathbb{Z}/N\mathbb{Z})^\times)$, then the
assertion follows from Lemmas \ref{coefficient} and \ref{determine}.
\par
Let $\alpha\in S_N$ with a lifting $\alpha_0=\left[\begin{matrix}A_0
& B_0\\C_0&D_0\end{matrix}\right]$ to
$\mathrm{Sp}_{2g}(\mathbb{Z})$. Note that since $N$ is even,
$\{{^t}A_0C_0\},\{{^t}B_0D_0\}\in2\mathbb{Z}^g$.
 Thus we derive that
\begin{eqnarray*}
&&(\Phi_{\left[\begin{smallmatrix}\mathbf{r}\\\mathbf{s}\end{smallmatrix}\right]}(Z)^{2N^2}
)^\alpha\\&=&(\Phi_{\left[\begin{smallmatrix}\mathbf{r}\\\mathbf{s}\end{smallmatrix}\right]}(Z)^{2N^2}
)^{\alpha_0}\quad \textrm{by Proposition
\ref{Groupaction}(iii)}\\
&=&
\Phi_{\left[\begin{smallmatrix}\mathbf{r}\\\mathbf{s}\end{smallmatrix}\right]}(Z)^{2N^2}\circ
\alpha_0\quad\textrm{by Proposition \ref{Groupaction}(ii)}\\
&=&(e(({^t}\mathbf{r}\mathbf{s}-{^t}\mathbf{r}'\mathbf{s}')/2)\Phi_{{^t}\alpha_0\left[\begin{smallmatrix}\mathbf{r}\\\mathbf{s}\end{smallmatrix}\right]}(Z))^{2N^2}
\quad\textrm{where $\left[\begin{matrix}
\mathbf{r}'\\
\mathbf{s}'
\end{matrix}\right]
={^t}\alpha_0 \left[\begin{matrix}
\mathbf{r}\\
\mathbf{s}
\end{matrix}\right]$, by Proposition \ref{transf}}
\\
&=&\Phi_{{^t}\alpha\left[\begin{smallmatrix}\mathbf{r}\\\mathbf{s}\end{smallmatrix}\right]}(Z)^{2N^2}
\quad\textrm{by Lemma \ref{determine}}.
\end{eqnarray*}
Since $G_N=\langle
\iota((\mathbb{Z}/N\mathbb{Z})^\times),S_N\rangle$, we get the
theorem.
\end{proof}

\begin{remark}\label{oddtransf}
\begin{itemize}
\item[(i)] This theorem tells us that the group $G_N$ acts on the family

$\{\Phi_{\left[\begin{smallmatrix}\mathbf{r}\\\mathbf{s}\end{smallmatrix}\right]}(Z)^{2N^2}\}
_{\mathbf{r},\mathbf{s}\in(1/N)\mathbb{Z}^g/\mathbb{Z}^g}$ in a
natural way.
\item[(ii)]
For later use we consider the case where
$\mathbf{r},\mathbf{s}\in(1/M)\mathbb{Z}^g$ for an odd positive
integer $M$. Then
$\Phi_{\left[\begin{smallmatrix}\mathbf{r}\\\mathbf{s}\end{smallmatrix}\right]}(Z)$
(respectively,
$\Phi_{\left[\begin{smallmatrix}\mathbf{r}\\\mathbf{s}\end{smallmatrix}\right]}(Z)^M$)
is of level $2M^2$ (respectively, $2M$) by Theorem \ref{modularity}.
As in the proof of Theorem \ref{family} one can show in a similar
way that if $\alpha\in G_{2M}$, then
\begin{equation*}
(\Phi_{\left[\begin{smallmatrix}\mathbf{r}\\\mathbf{s}\end{smallmatrix}\right]}(Z)^{2M^2}
)^\alpha=
\Phi_{{^t}\alpha\left[\begin{smallmatrix}\mathbf{r}\\\mathbf{s}\end{smallmatrix}\right]}(Z)^{2M^2}.
\end{equation*}
Now, let $\alpha\in G_{2M^2}$ with $a=\nu(\alpha)$ and set
$\gamma=\iota(a)\alpha\in S_{2M^2}$. Then we achieve that
\begin{eqnarray*}
\Phi_{\left[\begin{smallmatrix}\mathbf{r}\\\mathbf{s}\end{smallmatrix}\right]}(Z)^\alpha&=&
\Phi_{\left[\begin{smallmatrix}\mathbf{r}\\\mathbf{s}\end{smallmatrix}\right]}(Z)^{\iota(a)^{-1}\gamma}\\
&=&\Phi_{\left[\begin{smallmatrix}\mathbf{r}\\a\mathbf{s}\end{smallmatrix}\right]}(Z)^\gamma\quad
\textrm{by Lemmas \ref{coefficient} and \ref{determine}}\\
&=&e(({^t}\mathbf{r}a\mathbf{s}-{^t}\mathbf{r}'\mathbf{s}')/2)
\Phi_{\left[\begin{smallmatrix}\mathbf{r}'\\\mathbf{s}'\end{smallmatrix}\right]}(Z)
\quad\textrm{where}~\left[\begin{matrix}
\mathbf{r}'\\
\mathbf{s}'
\end{matrix}\right]
={^t}\gamma \left[\begin{matrix}
\mathbf{r}\\
a\mathbf{s}
\end{matrix}\right]=
{^t}\gamma{^t}\iota(a^{-1}) \left[\begin{matrix}
\mathbf{r}\\
\mathbf{s}
\end{matrix}\right]={^t}\alpha\left[\begin{matrix}
\mathbf{r}\\
\mathbf{s}
\end{matrix}\right]
, \\
&&\textrm{by Proposition \ref{transf} and Lemma \ref{determine}}\\
&=&e(({^t}\mathbf{r}a\mathbf{s}-{^t}\mathbf{r}'\mathbf{s}')/2)
\Phi_{{^t}\alpha\left[\begin{smallmatrix}\mathbf{r}\\\mathbf{s}\end{smallmatrix}\right]}(Z).
\end{eqnarray*}
\end{itemize}
\end{remark}

\begin{lemma}\label{conjugation}
Let $Z_0\in\mathfrak{H}_g$. If
$\Theta(\mathbf{0},Z_0;\mathbf{0},\mathbf{0})$ is nonzero, then we
have
\begin{equation*}
\overline{\Phi_{\left[\begin{smallmatrix}\mathbf{r}\\\mathbf{s}\end{smallmatrix}\right]}(Z_0)}
=
\Phi_{\left[\begin{smallmatrix}\mathbf{r}\\-\mathbf{s}\end{smallmatrix}\right]}(-\overline{Z}_0),
\end{equation*}
where the bar indicates the usual complex conjugation.
\end{lemma}
\begin{proof}
Consider the expansion
$\Phi_{\left[\begin{smallmatrix}\mathbf{r}\\\mathbf{s}\end{smallmatrix}\right]}(Z)
=\sum_{\xi}c(\xi)e(\mathrm{tr}(\xi
Z)/2N^2)/\Theta(\mathbf{0},Z;\mathbf{0},\mathbf{0})$ with $c(\xi)\in
\mathbb{Q}(\zeta_{2N^2})$, where $\xi$ runs over all $g\times g$
positive semi-definite symmetric matrices over half integers with
integral diagonal entries. Note that
$\Theta(\mathbf{0},Z;\mathbf{0},\mathbf{0})$ has rational Fourier
coefficients. Then we obtain that
\begin{eqnarray*}
\overline{\Phi_{\left[\begin{smallmatrix}\mathbf{r}\\\mathbf{s}\end{smallmatrix}\right]}(Z_0)}
&=&\sum_{\xi}\overline{c(\xi)}e(\mathrm{tr}(\xi(-\overline{Z}_0))/2N^2)/\Theta(\mathbf{0},-\overline{Z}_0;\mathbf{0},\mathbf{0})
\quad\textrm{because}~\overline{e(z)}=e(-\overline{z})~(z\in\mathbb{C})\\
&=&\sum_{\xi}c(\xi)^{\rho(-1)}e(\mathrm{tr}(\xi(-\overline{Z}_0))/2N^2)/\Theta(\mathbf{0},-\overline{Z}_0;\mathbf{0},\mathbf{0})\\
&&\textrm{where $\rho(-1)$ is the endomorphism of
$\mathbb{Q}(\zeta_{2N^2})$ induced from
$\zeta_{2N^2}\mapsto\zeta_{2N^2}^{-1}=\overline{\zeta}_{2N^2}$}\\
&=&(\sum_{\xi}c(\xi)e(\mathrm{tr}(\xi
Z)/2N^2)/\Theta(\mathbf{0},Z;\mathbf{0},\mathbf{0}))^{\iota(-1)^{-1}}(-\overline{Z}_0)\quad\textrm{by
Proposition \ref{Groupaction}(i)}\\
&=&\Phi_{\left[\begin{smallmatrix}\mathbf{r}\\\mathbf{s}\end{smallmatrix}\right]}(Z)^{\iota(-1)^{-1}}(-\overline{Z}_0)\\
&=&\Phi_{\left[\begin{smallmatrix}\mathbf{r}\\-\mathbf{s}\end{smallmatrix}\right]}(-\overline{Z}_0)
\quad\textrm{by
Lemma \ref{coefficient}}.
\end{eqnarray*}
\end{proof}

\section {Shimura's reciprocity law}\label{Srl}

We shall briefly explain Shimura's reciprocity law which connects
the theory of Siegel modular functions with class field theory.
\par
Let $g$ be a positive integer and $K$ be a CM-field with
$[K:\mathbb{Q}]=2g$, that is, a totally imaginary quadratic
extension of a totally real algebraic number field. Fix a set
$\{\varphi_1,\ldots,\varphi_g\}$ of $g$ embeddings
$K\rightarrow\mathbb{C}$ such that no two of them are complex
conjugate. One can take an element $\xi\in K$ satisfying
\begin{itemize}
\item[(i)] $\xi$ is purely imaginary,
\item[(ii)] $-\xi^2$ is totally positive,
\item[(iii)] $\mathrm{Im}(\xi^{\varphi_k})>0$ for all
$k=1,\ldots,g$,
\item[(iv)] $\mathrm{Tr}_{K/\mathbb{Q}}(\xi x)\in\mathbb{Z}$ for
all $x\in\mathcal{O}_K$
\end{itemize}
(\cite[p.43]{Shimura}). Define a map $\Phi:K\rightarrow\mathbb{C}^g$
by
$\Phi(x)=\left[\begin{matrix} x^{\varphi_1}\\
\vdots\\
x^{\varphi_g}
\end{matrix}\right]$, and
let $L=\{\Phi(x)\in\mathbb{C}^g~|~x\in\mathcal{O}_K\}$ which forms a
lattice in $\mathbb{C}^g$. Then we have an $\mathbb{R}$-bilinear map
$E:\mathbb{C}^g\times\mathbb{C}^g\rightarrow\mathbb{R}$ defined by
\begin{equation*}
E(\mathbf{u},\mathbf{v})=\sum_{k=1}^g
\xi^{\varphi_k}(u_k\overline{v}_k-\overline{u}_k{v}_k)\quad(\mathbf{u}=\left[\begin{matrix}u_1\\
\vdots\\u_g\end{matrix}\right],~\mathbf{v}=\left[\begin{matrix}v_1\\
\vdots\\v_g\end{matrix}\right]).
\end{equation*}
And $E$ gives a non-degenerate Riemann form on the complex torus
$\mathbb{C}^g/L$ satisfying
$E(\Phi(x),\Phi(y))=\mathrm{Tr}_{K/\mathbb{Q}}(\xi
x\overline{y})\in\mathbb{Z}$ for all $x,y\in\mathcal{O}_K$
\cite[p.44]{Shimura}, which makes it an abelian variety. Hence one
can find a positive integer $\mu$, a diagonal matrix
$\mathcal{E}=\left[\begin{matrix}\varepsilon_1 &
&\\&\ddots&\\&&\varepsilon_g\end{matrix}\right]\in\mathrm{Mat}_g(\mathbb{Z})$
with $\varepsilon_1=1$ and $\varepsilon_k|\varepsilon_{k+1}$
($k=1,\ldots,g-1$), and a complex $g\times 2g$ matrix
$\Omega=\left[\begin{matrix}\omega_1 & \omega_2\end{matrix}\right]$
with $\omega_1,\omega_2\in\mathrm{Mat}_g(\mathbb{C})$ such that
\begin{itemize}
\item[(i)]
$E(\Omega\mathbf{x},\Omega\mathbf{y})=\mu{^t}\mathbf{x}J\mathbf{y}$
 for all $\mathbf{x},\mathbf{y}\in\mathbb{R}^{2g}$,
\item[(ii)]
$L=\bigg\{\Omega\left[\begin{matrix}\mathbf{u}\\\mathbf{v}\end{matrix}\right]~|~\mathbf{u}\in\mathbb{Z}^g,~\mathbf{v}\in\mathcal{E}\mathbb{Z}^g\bigg\}$
\end{itemize}
\cite[p. 675]{Shimura3}. It is well-known that
$Z_0=\omega_2^{-1}\omega_1$ lies in $\mathfrak{H}_g$. Let $K^*$ be
the reflex field of $K$, namely $K^*=\mathbb{Q}(\sum_{k=1}^g
x^{\varphi_k}|x\in K)$. As a consequence of the main theorem of
complex multiplication we have the following proposition.

\begin{proposition}\label{CM}
If $f$ is an element of $\mathcal{F}_N$ which is finite at $Z_0$,
then the special value $f(Z_0)$ belongs to the maximal abelian
extension of $K^*$.
\end{proposition}
\begin{proof}
See \cite[Theorem 26.6]{Shimura}.
\end{proof}

For simplicity and later use, we assume that $\mathcal{E}=I_{g}$. We
define, as a representation map, a ring homomorphism
$h:K\rightarrow\mathrm{Mat}_{2g}(\mathbb{Q})$ as follows: Fix
$\xi_1,\xi_2,\ldots,\xi_{2g}\in K$ such that
$\Omega=\left[\begin{matrix}
\Phi(\xi_1)&\Phi(\xi_2)&\cdots&\Phi(\xi_{2g})\end{matrix}\right]$
and let $x\in K$. If $x\xi_j=\sum_{k=1}^{2g}r_{jk}\xi_k$ with
$r_{jk}\in\mathbb{Q}$ ($j=1,\ldots,g$), then we define
\begin{equation*}
h(x)=\left[\begin{matrix}r_{jk}\end{matrix}\right]_{1\leq j,k\leq
2g}.
\end{equation*}
Since
$L=\Phi(\mathcal{O}_K)=\Omega\mathbb{Z}^{2g}=\mathbb{Z}\Phi(\xi_1)+\cdots+
\mathbb{Z}\Phi(\xi_{2g})$ by the assumption $\mathcal{E}=I_g$, the
set $\{\xi_1,\ldots,\xi_{2g}\}$ forms a basis of $\mathcal{O}_K$ as
 $\mathbb{Z}$-module. Hence, if $x\in\mathcal{O}_K$, then
$h(x)\in\mathrm{Mat}_{2g}(\mathbb{Z})$.
\par
On the other hand, we take a Galois extension $K'$ of $\mathbb{Q}$
containing $K$, and extend $\varphi_k$ ($k=1,\ldots,g$) to an
element of $\mathrm{Gal}(K'/\mathbb{Q})$. We then define a
homomorphism $\varphi^*:(K^*)^\times\rightarrow K^\times$ by
\begin{equation*}
\varphi^*(x)=\prod_{k=1}^g
x^{\varphi_k^{-1}}\quad(x\in(K^*)^\times).
\end{equation*}

\begin{proposition}[Shimura's reciprocity law]\label{reciprocity}
Let $f$ be as in \textup{Proposition \ref{CM}}. Take a positive
integer $M$ such that $N|M$ and $f(Z_0)\in K^*_{(M)}$. Let $x$ be an
element of $\mathcal{O}_{K^*}$ which is prime to $M$. Considering
$h(\varphi^*(x))$ as an element of
$\mathrm{GSp}_{2g}(\mathbb{Z}/M\mathbb{Z})$ we have
\begin{equation*}
f(Z_0)^{(\frac{K^*_{(M)}/K^*}{(x)})}=f^{h(\varphi^*(x))}(Z_0),
\end{equation*}
where $(\frac{K^*_{(M)}/K^*}{\cdot})$ is the Artin reciprocity map.
\end{proposition}
\begin{proof}
See \cite[Theorem 26.8]{Shimura} or \cite[$\S$2.7]{Shimura1}.
\end{proof}

\begin{remark}\label{liein}
Assume further $x\equiv1\pmod{N}$. It is obvious that
$\varphi^*(x)\equiv1\pmod{N}$, which yields $h(\varphi^*(x))\equiv
I_{2g}\pmod{N}$. Hence we get by Proposition \ref{reciprocity}
$f(Z_0)^{(\frac{K^*_{(M)}/K^*}{(x)})}=f^{h(\varphi^*(x))}(Z_0)=f(Z_0)$.
This means that $f(Z_0)$ lies in $K^*_{(N)}$.
\end{remark}

\section {Construction of class fields}\label{construction}

Let $K=\mathbb{Q}(\zeta)$ with $\zeta=\zeta_5$, which is a CM-field
of degree $2g=[K:\mathbb{Q}]=4$. In this section we shall examine
the subfield of $K_{(2p)}$ for an odd prime $p$ which is generated
by the special value of
$\Phi_{\left[\begin{smallmatrix}1/p\\0\\0\\0\end{smallmatrix}\right]}(Z)^{2p^2}$
by using Shimura's reciprocity law.
\par
Fix a set of two embeddings $\{\varphi_1,\varphi_2\}$, where
$\varphi_1$ and $\varphi_2$ are defined by
\begin{equation*}
\zeta^{\varphi_1}=\zeta\quad\textrm{and}\quad\zeta^{\varphi_2}=\zeta^2.
\end{equation*}
If we set $\xi=(\zeta-\zeta^4)/5$, then one can readily check that
$\xi$ satisfies the conditions (i)$\sim$(iv) in the beginning of
$\S$\ref{Srl}. Let $\Phi:K\rightarrow\mathbb{C}^2$ be the map given
by
$\Phi(x)=\left[\begin{matrix}x^{\varphi_1}\\x^{\varphi_2}\end{matrix}\right]$
and $L=\Phi(\mathcal{O}_K)$. We have an $\mathbb{R}$-bilinear map
$E:\mathbb{C}^2\times\mathbb{C}^2\rightarrow\mathbb{R}$ defined by
\begin{equation*}
E(\mathbf{u},\mathbf{v})=\xi^{\varphi_1}(u_1\overline{v}_1-\overline{u}_1{v}_1)
+\xi^{\varphi_2}(u_2\overline{v}_2-\overline{u}_2{v}_2)\quad(\mathbf{u}=\left[\begin{matrix}u_1\\u_2\end{matrix}\right],~
\mathbf{v}=\left[\begin{matrix}v_1\\v_2\end{matrix}\right]),
\end{equation*}
which induces a non-degenerate Riemann form on $\mathbb{C}^2/L$ by
$\S$\ref{Srl}. Set
\begin{equation*}
\xi_1=\zeta^2,~\xi_2=\zeta^4,~\xi_3=\zeta,~\xi_4=\zeta+\zeta^3,
\end{equation*}
and
\begin{equation*}
\Omega= \left[\begin{matrix}\Phi(\xi_1) & \Phi(\xi_2) & \Phi(\xi_3)
& \Phi(\xi_4)
\end{matrix}\right]=\left[\begin{matrix}\zeta^2 & \zeta^4 & \zeta & \zeta+\zeta^3\\
\zeta^4 & \zeta^3 & \zeta^2 & \zeta^2+\zeta\end{matrix}\right].
\end{equation*}
One can readily verify that
$\left[\begin{matrix}E(\Phi(\xi_j),\Phi(\xi_k))\end{matrix}\right]_{1\leq
j,k\leq 4}=J$, from which it follows that
$E(\Omega\mathbf{x},\Omega\mathbf{y})= {^t}\mathbf{x}J\mathbf{y}$
for all $\mathbf{x},\mathbf{y}\in\mathbb{R}^4$ because $E$ is
$\mathbb{R}$-bilinear. Furthermore, since
$\mathcal{O}_K=\mathbb{Z}[\zeta]=\xi_1\mathbb{Z}+\xi_2\mathbb{Z}+\xi_3\mathbb{Z}+\xi_4\mathbb{Z}$,
we get $L=\Phi(\mathcal{O}_K)=\Omega\mathbb{Z}^{4}$. (So, $\mu=1$
and $\mathcal{E}=I_2$ in $\S$\ref{Srl}.) Then we obtain a CM-point
\begin{equation*}
Z_0=\left[\begin{matrix}\Phi(\xi_3) &
\Phi(\xi_4)\end{matrix}\right]^{-1} \left[\begin{matrix}\Phi(\xi_1)
& \Phi(\xi_2)
\end{matrix}\right]=
\left[\begin{matrix}\zeta & \zeta+\zeta^3\\\zeta^2 &
\zeta^2+\zeta\end{matrix}\right]^{-1} \left[\begin{matrix}\zeta^2 &
\zeta^4\\\zeta^4 & \zeta^3
\end{matrix}\right]\quad(\in\mathfrak{H}_2).
\end{equation*}
Now we define a ring homomorphism
$h:K\rightarrow\mathrm{Mat}_4(\mathbb{Q})$ by the relation
\begin{equation*}
x\left[\begin{matrix}\xi_1\\\vdots\\\xi_4\end{matrix}\right]=h(x)\left[\begin{matrix}\xi_1\\\vdots\\\xi_4\end{matrix}\right]
\quad(x\in K).
\end{equation*}
Note that the reflex field $K^*$ of $K$ is the same as $K$. We
further define an endomorphism $\varphi^*$ of $K^\times$ by
\begin{equation*}
\varphi^*(x)=x^{\varphi_1^{-1}}x^{\varphi_2^{-1}}\quad(x\in
K^\times).
\end{equation*}
\par
Assume that $p$ is an odd prime.

\begin{lemma}\label{little}
Let $\mathbf{r},\mathbf{s}\in(1/p)\mathbb{Z}^2$.
\begin{itemize}
\item[\textup{(i)}] $\Theta(\mathbf{0},Z_0;\mathbf{0},\mathbf{0})$
is nonzero.
\item[\textup{(ii)}] $2\Phi_{\left[\begin{smallmatrix}\mathbf{r}\\\mathbf{s}\end{smallmatrix}\right]}(Z_0)$
is an algebraic integer.
\item[\textup{(iii)}] $\Phi_{\left[\begin{smallmatrix}\mathbf{r}\\\mathbf{s}\end{smallmatrix}\right]}(Z_0)
\in K_{(2p^2)}$ and
$\Phi_{\left[\begin{smallmatrix}\mathbf{r}\\\mathbf{s}\end{smallmatrix}\right]}(Z_0)^p
\in K_{(2p)}$.
\item[\textup{(iv)}] If $\mathbf{r}=\left[\begin{matrix}r_1\\r_2\end{matrix}\right]$ and
$\mathbf{s}=\pm\left[\begin{matrix}r_1-r_2\\-r_1\end{matrix}\right]$
for $r_1,r_2\in(1/p)\mathbb{Z}$, then
$e(-{^t}\mathbf{r}\mathbf{s}/2)\Phi_{\left[\begin{smallmatrix}\mathbf{r}\\\mathbf{s}\end{smallmatrix}\right]}(Z_0)$
is real.
\end{itemize}
\end{lemma}
\begin{proof}
(i) See \cite[p.784]{d-G}.\\
(ii) See \cite[Proposition 2]{Komatsu}.\\
(iii) This is immediate from Remarks \ref{oddtransf}(ii) and \ref{liein}.\\
(iv) We observe that
\begin{equation}\label{bar}
\overline{\Omega}=\left[\begin{matrix}\zeta^3 & \zeta & \zeta^4 &
\zeta^4+\zeta^2\\
\zeta & \zeta^2 & \zeta^3 & \zeta^3+\zeta^4\end{matrix}\right]
=\Omega\alpha\quad\textrm{where}~\alpha=\left[\begin{matrix}0 & 0 &
0 & 1\\
0 & 0 & 1 & 1\\
-1 & 1 & 0 & 0\\
1 & 0 & 0 & 0\end{matrix}\right].
\end{equation}
Let $\Omega=\left[\begin{matrix}\omega_1 &
\omega_2\end{matrix}\right]$ and $\alpha=\left[\begin{matrix}A &
B\\C & D\end{matrix}\right]$ with $2\times 2$ matrices
$\omega_1,\omega_2,A,B,C,D$. Then we derive that
\begin{eqnarray}
-\overline{Z}_0&=&-{^t}\overline{Z}_0\quad\textrm{since $Z_0={^t}Z_0$}\nonumber\\
&=&-{^t}\overline{\omega}_1{^t}\overline{\omega}_2^{-1}\nonumber\\
&=&-({^t}A{^t}\omega_1+{^t}C{^t}\omega_2)({^t}B{^t}\omega_1+{^t}D{^t}\omega_2)^{-1}\quad\textrm{by (\ref{bar})}\nonumber\\
&=&-({^t}A{^t}\omega_1+{^t}C{^t}\omega_2){^t}\omega_2^{-1}
{^t}\omega_2({^t}B{^t}\omega_1+{^t}D{^t}\omega_2)^{-1}\nonumber\\
&=&-({^t}A{^t}\omega_1{^t}\omega_2^{-1}+{^t}C)
({^t}B{^t}\omega_1{^t}\omega_2^{-1}+{^t}D)^{-1}\nonumber\\
&=&-({^t}AZ_0+{^t}C) ({^t}B{^t}Z_0+{^t}D)^{-1}\quad\textrm{because
$Z_0={^t}Z_0$}\nonumber\\
&=&\beta(Z_0)\quad\textrm{where}~\beta=
\left[\begin{matrix}-{^t}A & -{^t}C\\
{^t}B & {^t}D\end{matrix}\right]=\left[\begin{matrix}0 & 0 & 1 &
-1\\
0 & 0 & -1 & 0\\0 & 1 & 0 & 0\\
1 & 1 & 0 & 0\end{matrix}\right].\label{beta}
\end{eqnarray}
One can readily check that $\beta\in\mathrm{Sp}_4(\mathbb{Z})$.
Hence we achieve that
\begin{eqnarray*}
\overline{e(-{^t}\mathbf{r}\mathbf{s}/2)\Phi_{\left[\begin{smallmatrix}\mathbf{r}\\\mathbf{s}\end{smallmatrix}\right]}(Z_0)}&=&
e({^t}\mathbf{r}\mathbf{s}/2)\Phi_{\left[\begin{smallmatrix}\mathbf{r}\\-\mathbf{s}\end{smallmatrix}\right]}(-\overline{Z}_0)\quad
\textrm{by Lemma \ref{conjugation}}\\
&=&e({^t}\mathbf{r}\mathbf{s}/2)\Phi_{\left[\begin{smallmatrix}\mathbf{r}\\-\mathbf{s}\end{smallmatrix}\right]}(\beta(Z_0))\quad\textrm{by
(\ref{beta})}\\
&=&e({^t}\mathbf{r}\mathbf{s}/2)e(({^t}\mathbf{r}(-\mathbf{s})-{^t}\mathbf{r}'\mathbf{s}')/2)
\Phi_{{^t}\beta\left[\begin{smallmatrix}\mathbf{r}\\-\mathbf{s}\end{smallmatrix}\right]}(Z_0)\quad
\textrm{where}~\left[\begin{matrix}\mathbf{r}'\\\mathbf{s}'\end{matrix}\right]=
{^t}\beta\left[\begin{matrix}\mathbf{r}\\-\mathbf{s}\end{matrix}\right]
,\\&&\textrm{by
Remark \ref{oddtransf}(ii)}\\
&=&\left\{\begin{array}{ll}e((-r_1^2+2r_1r_2)/2)\Phi_{\left[\begin{smallmatrix}r_1\\r_2\\r_1-r_2\\-r_1\end{smallmatrix}\right]}(Z_0)
& \textrm{if}~\mathbf{s}=\left[\begin{matrix}r_1-r_2\\-r_1\end{matrix}\right],\\
e((r_1^2-2r_1r_2)/2)\Phi_{\left[\begin{smallmatrix}-r_1\\-r_2\\r_1-r_2\\-r_1\end{smallmatrix}\right]}(Z_0)
&
\textrm{if}~\mathbf{s}=\left[\begin{matrix}-r_1+r_2\\r_1\end{matrix}\right],
\end{array}\right.\\
&=&
e(-{^t}\mathbf{r}\mathbf{s}/2)\Phi_{\left[\begin{smallmatrix}\mathbf{r}\\\mathbf{s}\end{smallmatrix}\right]}(Z_0)\quad\textrm{by
(\ref{minus})},
\end{eqnarray*}
which ensures that
$e(-{^t}\mathbf{r}\mathbf{s}/2)\Phi_{\left[\begin{smallmatrix}\mathbf{r}\\\mathbf{s}\end{smallmatrix}\right]}(Z_0)$
is real.
\end{proof}

Now, we shall investigate the field $K(z^{2p^2})$ with
\begin{equation*}
z=\Phi_{\left[\begin{smallmatrix}1/p\\0\\0\\0\end{smallmatrix}\right]}(Z_0).
\end{equation*} Note that $z\in K_{(2p^2)}$ and $z^p\in K_{(2p)}$ by
Lemma \ref{little}(iii). Let \begin{equation*}
x_1=1+2p\zeta,~x_2=1+2p(\zeta^2-\zeta^3+\zeta^4). \end{equation*}
Then we get \begin{eqnarray} &&h(\varphi^*(x_1))\equiv
h(1+2p(\zeta+\zeta^3)))\equiv\left[\begin{matrix} 1-2p & -2p & -2p &
0\\ 0 & 1-2p & 0 & -2p\\ 2p & 2p & 1 & 0\\ 2p & 4p & 2p &
1\end{matrix}\right]\pmod{2p^2},\label{h2}\\
&&h(\varphi^*(x_2))\equiv h(1+2p(\zeta+2\zeta^2-\zeta^3))\equiv
\left[\begin{matrix}1+2p & 6p & -2p & 4p\\ -4p & 1-2p & 4p & -2p\\
2p & -2p & 1-4p & 4p\\ -2p & -4p & -6p & 1
\end{matrix}\right]\pmod{2p^2},\quad\quad\quad\label{h3} \end{eqnarray} and
\begin{equation}\label{wnu}
\nu(h(\varphi^*(x_1)))\equiv\nu(h(\varphi^*(x_2)))\equiv1-2p\pmod{2p^2}.
\end{equation} Here we observe that $h(\varphi^*(x_k))$ ($k=1,2$)
belongs to $G_{2p^2}$.

\begin{lemma}\label{2p^2equal}
Let $a,b,c,d\in\mathbb{Z}$.
\begin{itemize}
\item[\textup{(i)}] We have the formula
\begin{eqnarray*}
&&\Phi_{\left[\begin{smallmatrix}a/p\\b/p\\c/p\\d/p\end{smallmatrix}\right]}(Z_0)^{(\frac{K_{(2p^2)}/K}{(x_k)})}
\quad(k=1,2)\\
&=&\left\{\begin{array}{ll}e((-a^2+2ad-b^2-c^2-2cd-2d^2)/p)
\Phi_{\left[\begin{smallmatrix}a/p\\b/p\\c/p\\d/p\end{smallmatrix}\right]}(Z_0) & \textrm{if}~k=1,\\
e((-a^2+4ab-4ac-b^2+4bc-c^2+2cd+2d^2)/p)
\Phi_{\left[\begin{smallmatrix}a/p\\b/p\\c/p\\d/p\end{smallmatrix}\right]}(Z_0)
& \textrm{if}~k=2.
\end{array}\right.
\end{eqnarray*}
\item[\textup{(ii)}] $\zeta_{2p^2}^{(\frac{K_{(2p^2)}/K}{(x_k)})}=\zeta_{2p^2}^{1-2p}$ \textup{(}$k=1,2$\textup{)}.
\item[\textup{(iii)}]
Assume that
$\Phi_{\left[\begin{smallmatrix}1/p\\0\\0\\0\end{smallmatrix}\right]}(Z_0)$
and
$\Phi_{\left[\begin{smallmatrix}a/p\\b/p\\c/p\\d/p\end{smallmatrix}\right]}(Z_0)$
are nonzero. If
$\Phi_{\left[\begin{smallmatrix}1/p\\0\\0\\0\end{smallmatrix}\right]}(Z_0)^{2p^2}
=\Phi_{\left[\begin{smallmatrix}a/p\\b/p\\c/p\\d/p\end{smallmatrix}\right]}(Z_0)^{2p^2}$,
then $-2ab+2ac+ad-2bc-2cd-2d^2\equiv0\pmod{p}$.
\end{itemize}
\end{lemma}
\begin{proof}
(i) We derive that
\begin{eqnarray*}
&&\Phi_{\left[\begin{smallmatrix}a/p\\b/p\\c/p\\d/p\end{smallmatrix}\right]}(Z_0)^{(\frac{K_{(2p^2)}/K}{(x_k)})}
\quad(k=1,2)\\
&=&\Phi_{\left[\begin{smallmatrix}a/p\\b/p\\c/p\\d/p\end{smallmatrix}\right]}(Z)^{h(\varphi^*(x_k))}(Z_0)
\quad\textrm{by Proposition \ref{reciprocity}}\\
&=&e(({^t}\mathbf{r}(1-2p)\mathbf{s}-{^t}\mathbf{r}'\mathbf{s}')/2)
\Phi_{\left[\begin{smallmatrix}\mathbf{r}'\\\mathbf{s}'\end{smallmatrix}\right]}(Z_0)
\quad\textrm{where}~\mathbf{r}=\left[\begin{matrix}a/p\\b/p\end{matrix}\right],~
\mathbf{s}=\left[\begin{matrix}c/p\\d/p\end{matrix}\right]~\textrm{and}~
\left[\begin{matrix}\mathbf{r}'\\\mathbf{s}'\end{matrix}\right]=
{^t}h(\varphi^*(x_k))\left[\begin{matrix}\mathbf{r}\\\mathbf{s}\end{matrix}\right],\\
&&\textrm{by Remark \ref{oddtransf}(ii) and (\ref{wnu})}\\
&=&\left\{\begin{array}{ll}e((a^2+b^2-c^2-2cd-2d^2)/p)
\Phi_{\left[\begin{smallmatrix}a/p\\b/p\\c/p\\d/p\end{smallmatrix}\right]
+\left[\begin{smallmatrix}-2a+2c+2d\\-2a-2b+2c+4d\\
-2a+2d\\-2b
\end{smallmatrix}\right]
}(Z_0) & \textrm{if}~k=1,\\
e((a^2-4ab+b^2-c^2+2cd+2d^2)/p)
\Phi_{\left[\begin{smallmatrix}a/p\\b/p\\c/p\\d/p\end{smallmatrix}\right]
+\left[\begin{smallmatrix}2a-4b+2c-2d\\
6a-2b-2c-4d\\-2a+4b-4c-6d\\
4a-2b+4c
\end{smallmatrix}\right]
}(Z_0) & \textrm{if}~k=2,
\end{array}\right.\\
&&\textrm{by (\ref{h2}), (\ref{h3}) and Lemma \ref{determine}}\\
&=&\left\{\begin{array}{ll}e((-a^2+2ad-b^2-c^2-2cd-2d^2)/p)
\Phi_{\left[\begin{smallmatrix}a/p\\b/p\\c/p\\d/p\end{smallmatrix}\right]}(Z_0) & \textrm{if}~k=1,\\
e((-a^2+4ab-4ac-b^2+4bc-c^2+2cd+2d^2)/p)
\Phi_{\left[\begin{smallmatrix}a/p\\b/p\\c/p\\d/p\end{smallmatrix}\right]}(Z_0)
& \textrm{if}~k=2,
\end{array}\right.\\
&&\textrm{by Proposition \ref{transl}}.
\end{eqnarray*}
(ii) For $k=1,2$ we see that
\begin{eqnarray*}
\zeta_{2p^2}^{(\frac{K_{(2p^2)}/K}{(x_k)})}&=&\zeta_{2p^2}^{h(\varphi^*(x_k))}\quad\textrm{by
Proposition \ref{reciprocity}}\\
&=&\zeta_{2p^2}^{\nu(h(\varphi^*(x_k)))}\quad\textrm{by
Proposition \ref{Groupaction}}\\
&=&\zeta_{2p^2}^{1-2p}\quad\textrm{by (\ref{wnu})}.
\end{eqnarray*}
(iii) If
$\Phi_{\left[\begin{smallmatrix}1/p\\0\\0\\0\end{smallmatrix}\right]}(Z_0)^{2p^2}
=\Phi_{\left[\begin{smallmatrix}a/p\\b/p\\c/p\\d/p\end{smallmatrix}\right]}(Z_0)^{2p^2}$,
then we deduce that
\begin{eqnarray*}
&&\Phi_{\left[\begin{smallmatrix}1/p\\0\\0\\0\end{smallmatrix}\right]}(Z_0)
^{(\frac{K_{(2p^2)}/K}{(x_1)})}/
\Phi_{\left[\begin{smallmatrix}1/p\\0\\0\\0\end{smallmatrix}\right]}(Z_0)
^{(\frac{K_{(2p^2)}/K}{(x_2)})}\\&=&1\quad\textrm{by (i)}\\
&=&(\xi\Phi_{\left[\begin{smallmatrix}a/p\\b/p\\c/p\\d/p\end{smallmatrix}\right]}(Z_0))
^{(\frac{K_{(2p^2)}/K}{(x_1)})}/
(\xi\Phi_{\left[\begin{smallmatrix}a/p\\b/p\\c/p\\d/p\end{smallmatrix}\right]}(Z_0))
^{(\frac{K_{(2p^2)}/K}{(x_2)})}\quad
\textrm{for some $2p^2$th root of unity $\xi$}\\
&=&\xi^{1-2p}\Phi_{\left[\begin{smallmatrix}a/p\\b/p\\c/p\\d/p\end{smallmatrix}\right]}(Z_0)
^{(\frac{K_{(2p^2)}/K}{(x_1)})}/
\xi^{1-2p}\Phi_{\left[\begin{smallmatrix}a/p\\b/p\\c/p\\d/p\end{smallmatrix}\right]}(Z_0)
^{(\frac{K_{(2p^2)}/K}{(x_2)})}\quad \textrm{by (ii)}\\
&=&e((-4ab+4ac+2ad-4bc-4cd-4d^2)/p)\quad\textrm{by (i)}.
\end{eqnarray*}
This proves (iii).
\end{proof}

Now we let
\begin{equation*}
T=\{\sigma\in\mathrm{Gal}(K_{(2p)}/K)~|~
(z^{\sigma'})^{(\frac{K_{(2p^2)}/K}{(x_1)})}=(z^{\sigma'})^{(\frac{K_{(2p^2)}/K}{(x_2)})}~
\textrm{for some extension $\sigma'$ of $\sigma$ to $K_{(2p^2)}$}\}.
\end{equation*}
Since $K_{(2p^2)}/K$ is an abelian extension, $T$ is obviously a
subgroup of $\mathrm{Gal}(K_{(2p)}/K)$. Furthermore, if $\sigma\in
T$, then
$(z^{\sigma'})^{(\frac{K_{(2p^2)}/K}{(x_1)})}=(z^{\sigma'})^{(\frac{K_{(2p^2)}/K}{(x_2)})}$
for all extensions $\sigma'$ of $\sigma$ to $K_{(2p^2)}$. Indeed,
let $\sigma$ be an element of $T$ with some extension $\sigma'$
satisfying
$(z^{\sigma'})^{(\frac{K_{(2p^2)}/K}{(x_1)})}=(z^{\sigma'})^{(\frac{K_{(2p^2)}/K}{(x_2)})}$.
Now that $z^p\in K_{(2p)}$, if $\sigma''$ is another extension of
$\sigma$ to $K_{(2p^2)}$, then $(z^p)^{\sigma''}=(z^p)^{\sigma'}$.
So we get $z^{\sigma''}=\xi z^{\sigma'}$ for some $p$th root of
unity $\xi$. And, since $\xi^{(\frac{K_{(2p^2)}/K}{(x_1)})}=
\xi^{(\frac{K_{(2p^2)}/K}{(x_2)})}=\xi^{1-2p}$ ($k=1,2$) by Lemma
\ref{2p^2equal}(ii), we obtain that
\begin{equation*}
(z^{\sigma''})^{(\frac{K_{(2p^2)}/K}{(x_1)})}= (\xi
z^{\sigma'})^{(\frac{K_{(2p^2)}/K}{(x_1)})}= (\xi
z^{\sigma'})^{(\frac{K_{(2p^2)}/K}{(x_2)})}=
(z^{\sigma''})^{(\frac{K_{(2p^2)}/K}{(x_2)})}.
\end{equation*}
\par
Let $F$ be the subfield of $K_{(2p)}$ fixed by $T$. Then by the
Galois theory we see the isomorphism
$\mathrm{Gal}(K_{(2p)}/F)\simeq T$.

\begin{theorem}\label{howlarge}
$K(z^{2p^2})$ contains $F$.
\end{theorem}
\begin{proof}
By the Galois theory the assertion is equivalent to saying that
$\mathrm{Gal}(K_{(2p)}/K(z^{2p^2}))$ is a subgroup of $T$. Let
$\sigma\in\mathrm{Gal}(K_{(2p)}/K(z^{2p^2}))$, that is, $\sigma$ is
an element of $\mathrm{Gal}(K_{(2p)}/K)$ which fixes $z^{2p^2}$. If
$\sigma'$ is an extension of $\sigma$ to $K_{(2p^2)}$, then we have
$z^{\sigma'}=\xi z$ for some $2p^2$th root of unity. Since both
$(\frac{K_{(2p^2)}/K}{(x_1)})$ and $(\frac{K_{(2p^2)}/K}{(x_2)})$
send $\xi z$ to $\xi^{1-2p}\zeta_p^{-1}z$ by Lemma
\ref{2p^2equal}(i) and (ii), $\sigma$ belongs to $T$. This proves
the theorem.
\end{proof}

\begin{theorem}\label{belong}
Let $x=a_0+a_1\zeta+a_2\zeta^2+a_3\zeta^3+a_4\zeta^4$ for integers
$a_0,a_1,a_2,a_3,a_4$ such that $\mathrm{N}_{K/\mathbb{Q}}(x)$ is
prime to $2p$ and $h(\varphi^*(x))\in G_{2p^2}$. Assume that $z$ is
nonzero. If $(\frac{K_{(2p)}/K}{(x)})$ fixes $z^{2p^2}$, then we
have $-2ab+2ac+ad-2bc-2cd-2d^2\equiv0\pmod{p}$, where
\begin{eqnarray*}
a&=&a_0^2-a_0a_1-a_0a_3+a_1a_2+a_1a_3-a_1a_4-a_2^2+a_2a_4,\\
b&=&-a_0a_1+a_0a_2-a_0a_3+a_0a_4+a_1a_2-a_2^2+a_3^2-a_3a_4,\\
c&=&-a_0a_1-a_0a_2+a_0a_3+a_0a_4+a_1^2-a_1a_3+a_2a_4-a_4^2,\\
d&=&a_0a_2-a_0a_3+a_1a_3-a_1a_4-a_2^2+a_2a_3-a_3a_4+a_4^2.
\end{eqnarray*}
\end{theorem}
\begin{proof}
One can readily verify that the first row of $h(\varphi^*(x))$ is
$\left[\begin{matrix}a & b & c &d\end{matrix}\right]$, namely
\begin{equation*}
x^{\varphi_1^{-1}}x^{\varphi_2^{-1}}\xi_1=
(a_0+a_1\zeta+a_2\zeta^2+a_3\zeta^3+a_4\zeta^4)
(a_0+a_1\zeta^3+a_2\zeta+a_3\zeta^4+a_4\zeta^2)\zeta^2=
a\xi_1+b\xi_2+c\xi_3+d\xi_4.
\end{equation*}
And we have
\begin{eqnarray*}
(z ^{2p^2})^{(\frac{K_{(2p)}/K}{(x)})}&=&
(\Phi_{\left[\begin{smallmatrix}1/p\\0\\0\\0\end{smallmatrix}\right]}(Z)^{2p^2})^{h(\varphi^*(x))}(Z_0)
\quad\textrm{by Proposition \ref{reciprocity}}\\
&=&\Phi_{{^t}h(\varphi^*(x))\left[\begin{smallmatrix}1/p\\0\\0\\0\end{smallmatrix}\right]}(Z)^{2p^2}(Z_0)
\quad\textrm{by Remark \ref{oddtransf}(ii)}\\
&=&\Phi_{\left[\begin{smallmatrix}a/p\\b/p\\c/p\\d/p\end{smallmatrix}\right]}(Z_0)^{2p^2}.
\end{eqnarray*}
We then conclude the theorem by Lemma \ref{2p^2equal}(iii).
\end{proof}

\begin{remark}
Let $x$ be of the form $1+2y$ for some $y\in\mathcal{O}_K$ such that
$\mathrm{N}_{K/\mathbb{Q}}(x)$ is prime to $2p$. Since
\begin{equation*}
h(\varphi^*(x))=h((1+2y)(1+2y)^{\varphi_2^{-1}})\equiv h(1)\equiv
I_{2g}\pmod{2},
\end{equation*}
$h(\varphi^*(x))$ belongs to $G_{2p^2}$ automatically.
\end{remark}

\begin{example}
We follow the notations and assumption in Theorem \ref{belong}. Let
$x=1+2(\zeta+\zeta^2)$, then $\mathrm{N}_{K/\mathbb{Q}}(x)=5$. So we
assume $p\neq5$. Since $a=-1$, $b=0$, $c=0$, $d=-2$, we get
$-2ab+2ac+ad-2bc-2cd-2d^2\equiv-6\pmod{p}$. Therefore, if $p\neq3$,
then $(\frac{K_{(2p)}/K}{(x)})$ does not fix $z^{2p^2}$ by Theorem
\ref{belong}.
\end{example}

\section {Remarks on primitive generators}

In this section we shall introduce two useful methods of combining
two generators of an abelian extension to get a primitive one. We
begin with an example.

\begin{example}
Following the same notations as in $\S$\ref{construction} we
consider the field $K(z^{p},\zeta_{25})$. Note that $(2z)^{p}$ is an
algebraic integer which lies in $K_{(2p)}$ by Lemma \ref{little}(ii)
and (iii). The prime ideal $(1-\zeta_{5})\mathcal{O}_K$ of $K$,
which lies above the prime ideal $5\mathbb{Z}$ of $\mathbb{Q}$, is
totally ramified in the extension $K(\zeta_{25})/K$ with
ramification index
$[K(\zeta_{25}):K]=[\mathbb{Q}(\zeta_{25}):\mathbb{Q}]/[\mathbb{Q}(\zeta_{5}):\mathbb{Q}]=5$.
On the other hand, all prime ideals of $K$ which are ramified in
$K_{(2p)}/K$ must divide the ideal $2p\mathcal{O}_K$. So we have
$[K(z^{p},\zeta_{25}):K(z^{p})]=5$. Furthermore, a conjugate of
$\zeta_{25}$ over $K(z^{p})$ is of the form $\zeta_{25}^{1+5k}$
($k=0,1,2,3,4$). And we get
\begin{eqnarray*}
\mathrm{Tr}_{K(z^{p},\zeta_{25})/K(z^{p})}(\zeta_{25})&=&\sum_{k=0}^4
\zeta_{25}^{1+5k}~=~ \zeta_{25}\sum_{k=0}^4\zeta_5^k~=~
0,\\
\mathrm{N}_{K(z^{p},\zeta_{25})/K(z^{p})}(3\zeta_{25}+1)&=&\prod_{k=0}^4
(3\zeta_{25}^{1+5k}+1)~=~
(-3\zeta_{25})^5\prod_{k=0}^4(-\zeta_5^k+(-3\zeta_{25})^{-1})~=~
243\zeta_5+1.
\end{eqnarray*}
 We then obtain
primitive generators of $K(z^{p},\zeta_{25})$ over $K$ by Theorems
\ref{primitive1} and \ref{primitive2} as follows:
\begin{equation*}
K(z^{p},\zeta_{25})=K(z^{p}+\zeta_{25})=K((3(2z)^{p}+1)(3\zeta_{25}+1)^{-5}
(243\zeta_5+1)).
\end{equation*}
\end{example}

\begin{theorem}\label{primitive1}
Let $L$ be an abelian extension of a number field $K$. Suppose that
$L=K(x,y)$ for some $x,y\in L$. Let $\ell=[L:K(x)]$, and $a$ and $b$
be any nonzero elements of $K$. Then we have
\begin{equation*}
L=K(ax+b(\ell y-\mathrm{Tr}_{L/K(x)}(y))).
\end{equation*}
\end{theorem}
\begin{proof}
If we set $\varepsilon=ax+b(\ell y-\mathrm{Tr}_{L/K(x)}(y))$, then
we achieve
\begin{eqnarray}
\mathrm{Tr}_{L/K(x)}(\varepsilon)&=&\mathrm{Tr}_{L/K(x)}(ax)+\mathrm{Tr}_{L/K(x)}(b\ell y)+\mathrm{Tr}_{L/K(x)}(-b\mathrm{Tr}_{L/K(x)}(y))\nonumber\\
&=&ax\mathrm{Tr}_{L/K(x)}(1)+b\ell \mathrm{Tr}_{L/K(x)}(y)-b\mathrm{Tr}_{L/K(x)}(y)\mathrm{Tr}_{L/K(x)}(1)\nonumber\\
&=&ax\ell.\label{tr}
\end{eqnarray}
Observe that  since $L/K$ is an abelian extension, any intermediate
field of $L/K$ is an abelian extension of $K$ by Galois theory. This
fact implies that $K(\varepsilon)$ contains $\varepsilon^\sigma$ for
any $\sigma\in\mathrm{Gal}(L/K)$; whence
$\mathrm{Tr}_{L/K(x)}(\varepsilon)$ belongs to $K(\varepsilon)$.
Therefore we deduce that
\begin{eqnarray*}
K(\varepsilon)&=&K(\varepsilon)(\mathrm{Tr}_{L/K(x)}(\varepsilon))~=~K(\mathrm{Tr}_{L/K(x)}(\varepsilon))(\varepsilon)\\
&=&K(ax\ell)(\varepsilon)~=~K(x)(\varepsilon)\quad\textrm{by (\ref{tr})}\\
&=&K(x)(\varepsilon-ax+b\mathrm{Tr}_{L/K(x)}(y))~=~K(x)(b\ell
y)~=~K(x)(y)~=~L.
\end{eqnarray*}
This completes the proof.
\end{proof}

\begin{lemma}\label{power}
Let $K$ be a number field and $x$ be an algebraic integer. Suppose
that $K(x)/K$ is a Galois extension. If $a$ and $b$ are nonzero
integers such that $2<|a/b|$, then we have
\begin{equation*}
K(x)=K((ax+b)^n)\quad\textrm{for a nonzero integer $n$}.
\end{equation*}
\end{lemma}
\begin{proof}
We first note that $ax+b$ is nonzero. Otherwise, $x=-b/a$ is a
nonzero rational number less than $1/2$, which contradicts the fact
that $x$ is an algebraic integer. Suppose on the contrary
$K(x)\supsetneq K((ax+b)^n)$. Then there exists a nontrivial element
$\sigma$ of $\mathrm{Gal}(K(x)/K((ax+b)^n))$. That is,
$((ax+b)^n)^\sigma=(ax+b)^n$ but $x^\sigma\neq x$, from which we see
that $ax^\sigma+b=\xi(ax+b)$ for some $|n|$th root of unity $\xi$
($\neq1$). Thus we have
\begin{equation}\label{zeta}
a(x^\sigma-\xi x)=b(\xi-1)~(\neq0).
\end{equation}
Let $\ell=[K(x):\mathbb{Q}]$. Since $x^\sigma-\xi x$ is a nonzero
algebraic integer, we derive that
\begin{equation*}
|\mathrm{N}_{K(x)/\mathbb{Q}}(a(x^\sigma-\xi
x))|=|a|^\ell|\mathrm{N}_{K(x)/\mathbb{Q}}(x^\sigma-\xi x)| \geq
|a|^\ell.
\end{equation*}
On the other hand, since any conjugate of $\xi-1$ over $\mathbb{Q}$
is of the form $\xi^t-1$ for some integer $t$ and $|\xi^t-1|\leq 2$,
we claim that
\begin{equation*}
|\mathrm{N}_{K(x)/\mathbb{Q}}(b(\xi-1))|
=|b|^\ell|\mathrm{N}_{K(x)/\mathbb{Q}}(\xi-1)| \leq|b|^\ell 2^\ell.
\end{equation*}
Hence we get from (\ref{zeta})  $|a|^\ell\leq|b|^\ell2^\ell$, which
contradicts the assumption $2<|a/b|$. Therefore we conclude the
lemma.
\end{proof}

\begin{theorem}\label{primitive2}
Let $L$ be an abelian extension of a number field $K$. Suppose
$L=K(x,y)$ for some algebraic integers $x$ and $y$. Let
$\ell=[L:K(x)]$ and $a$, $b$, $c$, $d$ be nonzero integers such that
$2<|a/b|$ and $2<|c/d|$. Then we have
\begin{equation*}
L=K((ax+b)^n(cy+d)^{-m\ell}\mathrm{N}_{L/K(x)}((cy+d)^m))
\quad\textrm{for nonzero integers $n$ and $m$}.
\end{equation*}
\end{theorem}
\begin{proof}
If we let
$\varepsilon=(ax+b)^n(cy+d)^{-m\ell}\mathrm{N}_{L/K(x)}((cy+d)^m)$,
then we get that
\begin{eqnarray}
\mathrm{N}_{L/K(x)}(\varepsilon)&=&\mathrm{N}_{L/K(x)}((ax+b)^n)
\mathrm{N}_{L/K(x)}((cy+d)^{-m\ell})
\mathrm{N}_{L/K(x)}(\mathrm{N}_{L/K(x)}((cy+d)^m))\nonumber\\
&=&\mathrm{N}_{L/K(x)}((cy+d)^{-m\ell})(\mathrm{N}_{L/K(x)}((cy+d)^m))^\ell\nonumber\\
&=&(ax+b)^{n\ell}.\label{norm}
\end{eqnarray}
Now that $L/K$ is an abelian extension, $K(\varepsilon)$
 contains $\mathrm{N}_{L/K(x)}(\varepsilon)$. Thus we obtain that
\begin{eqnarray*}
K(\varepsilon)&=&K(\varepsilon)(\mathrm{N}_{L/K(x)}(\varepsilon))~=~
K(\mathrm{N}_{L/K(x)}(\varepsilon))(\varepsilon)\\
&=&K((ax+b)^{n\ell})(\varepsilon)\quad\textrm{by (\ref{norm})}\\
&=&K(x)(\varepsilon)\quad\textrm{by Lemma \ref{power}}\\
&=&K(x)(\varepsilon/(ax+b)^n\mathrm{N}_{L/K(x)}((cy+d)^m)\\
&=&K(x)((cy+d)^{-m\ell}))\\
&=&K(x)(y)~=~L\quad\textrm{by Lemma \ref{power}}.
\end{eqnarray*}
This proves the theorem.
\end{proof}

\bibliographystyle{amsplain}

\address{
Department of Mathematical Sciences \\
KAIST \\
Daejeon 305-701 \\
Republic of Korea} {jkkoo@math.kaist.ac.kr}
\address{
Department of Mathematics\\
Hankuk University of Foreign Studies\\
Yongin-si, Gyeonggi-do 449-791\\
Republic of Korea } {dhshin@hufs.ac.kr}

\end{document}